\documentclass[12pt,a4paper]{amsart}
\usepackage{amssymb,enumerate}
\usepackage{tikz-cd}
\newcommand{\CC}{{\mathbb{C}}}
\newcommand{\FF}{{\mathbb{F}}}

\newcommand{\RR}{{\mathbb{R}}}

\newcommand{\ZZ}{{\mathbb{Z}}}

\newcommand{\cA} {\mathcal A}
\newcommand{\cB} {\mathcal B}
\newcommand{\cC} {\mathcal C}
\newcommand{\cE} {\mathcal E}
\newcommand{\cF} {\mathcal F}

\newcommand{\cL} {\mathcal L}

\newcommand{\cO} {\mathcal O}

\renewcommand{\S} {\mathcal S}
\newcommand{\cS} {\mathcal S}

\newcommand{\Fix}{{\operatorname{Fix}}}

\newcommand{\Hom}{{\operatorname{Hom}}}
\newcommand{\Ext}{{\operatorname{Ext}}}
\newcommand{\Id}{{\operatorname{Id}}}
\newcommand{\Ch}{{\operatorname{Ch}}}
\newcommand{\C}{{\operatorname{C}}}

\newcommand{\Inn}{{\operatorname{Inn}}}

\newcommand{\ob}{\operatorname{ob}}

\newcommand{\Rep}{{\operatorname{Rep}}}

\makeatletter
\newcommand{\colim@}[2]{
  \vtop{\m@th\ialign{##\cr
    \hfil$#1\operator@font colim$\hfil\cr
    \noalign{\nointerlineskip\kern1.5\ex@}#2\cr
    \noalign{\nointerlineskip\kern-\ex@}\cr}}
}
\newcommand{\colim}{
  \mathop{\mathpalette\colim@{\rightarrowfill@\textstyle}}\nmlimits@
}
\makeatother

\newcommand{\reg}{{\operatorname{reg}}}
\newcommand{\Aut}{{\operatorname{Aut}}}
\newcommand{\im}{{\operatorname{im}}}
\newcommand{\res}{{\operatorname{res}}}

\newcommand{\Spec}{{\operatorname{Spec}}}
\newcommand{\GL}{\operatorname{GL}}

\newcommand{\Mor}{{\operatorname{Mor}}}

\newtheorem{thm}{Theorem}[section]

\newtheorem{lem}[thm]{Lemma}
\newtheorem{cor}[thm]{Corollary}
\newtheorem{prop}[thm]{Proposition}
\newtheorem{conj}[thm]{Conjecture}

\theoremstyle{definition}

\newtheorem{defn}[thm]{Definition}

\numberwithin{equation}{section}

\begin{document}

\title{Spectra of subrings of cohomology generated by characteristic classes for fusion systems}

\author{Ian J. Leary}
\address{CGTA, School of Mathematical Sciences, University of Southampton, SO17 1BJ,
  U.K.}
\email{i.j.leary@soton.ac.uk}

\author{Jason Semeraro}
\address{Department of Mathematics, Loughborough University, LE11 3TT,
  U.K.}
\email{j.p.semeraro@lboro.ac.uk}

\begin{abstract}
If $\cF$ is a saturated fusion system on a finite $p$-group $S$, we define the Chern subring $\Ch(\cF)$ of $\cF$ to be the subring of $H^*(S;\FF_p)$ generated by Chern classes of $\cF$-stable representations of $S$. We show that $\Ch(\cF)$ is contained in $H^*(\cF;\FF_p)$ and apply a result of Green and the first author to describe its maximal ideal spectrum in terms of a certain category of elementary abelian subgroups. We obtain similar results for various related subrings, including those generated by characteristic classes of $\cF$-stable $S$-sets.
\end{abstract}

\keywords{fusion systems, Chern classes, support varieties}

\subjclass[2010]{20J06}

\date{\today}

\maketitle

\pagestyle{myheadings}
\markboth{IAN J. LEARY AND JASON SEMERARO}{SPECTRA OF SUBRINGS FOR FUSION SYSTEMS}

\section{Introduction}\label{s:intro}

Let $G$ be a finite group and $k$ be a field of characteristic~$p$. Quillen's description \cite{Q71} of the spectrum of the mod-$p$ cohomology ring of $G$ has been extremely useful in representation theory. The support variety of a $kG$-module
$M$ is a subvariety of the spectrum, subvarieties correspond to ideals
in the cohomology ring, and the ideal defining the support variety is
the kernel of the ring homomorphism $H^*(G)=\Ext_{\FF_p G}(\FF_p,\FF_p)\rightarrow
\Ext_{kG}(M,M)$. Just as ideals of $H^*(G)$ correspond to subvarieties
of the spectrum, subrings of $H^*(G)$ correspond to quotients of the
spectrum, with subrings over which $H^*(G)$ is integral corresponding to
quotients with finite fibres.  In~\cite{GL98} Green and the first named author gave a description of the spectra of subrings of $H^*(G)$ that
are both `large' and `natural', paying particular attention to the
subring generated by Chern classes of representations of $G$.  A later
article applied these results to the subring of $H^*(G)$ generated by
characteristic classes of homomorphisms
from $G$ to the symmetric group $\Sigma_n$ for all $n$~\cite{GLS02}.

Our aim is to extend many of these results concerning subrings of
$H^*(G)$ to analogous results for subrings of the cohomology of a
saturated fusion system.  Recall that a \textit{fusion system} $\cF$
on a finite $p$-group $S$ is a category with Ob$(\cF)=\{P \le S\}$ and
Mor$(\cF)$ a set of injective group homomorphisms between subgroups
satisfying some weak axioms. $\cF$ is \textit{saturated} if it
satisfies two additional `Sylow' axioms which hold whenever
$\cF=\cF_S(G)$ is the fusion system of a group $G$ with Sylow
$p$-subgroup $S$, where morphisms are given by $G$-conjugation
maps. The cohomology $H^*(\cF)$ of a saturated fusion system $\cF$ is
defined to be the subring of $\cF$-stable elements in $H^*(S)$ (see
Section \ref{ss:coh1}).

We now fix a finite $p$-group $S$ and let $\cF$ be a saturated fusion
system on $S$. We define categories of elementary abelian subgroups of
$S$ by stipulating that an injective homomorphism $f \in
\Hom(E_1,E_2)$ is in
\begin{itemize}
\item[] $\cE(\cF)$ iff there exists $\varphi \in \cF$ such that $f(e)=\varphi(e)$ for
  all $e \in E_1$;
\item[] $\cE'(\cF)$ iff for each  $e \in E_1$
 there exists  $\varphi \in \cF$ such that $f(e)=\varphi(e)$;
\item[] $\cE_\RR'(\cF)$ iff for each $e \in E_1$ there exists $\varphi
  \in \cF$ such that $f(e) \in \{\varphi(e), \varphi(e^{-1})\}$;
\item[] $\cE_P'(\cF)$ iff for all  $e \in E_1$, $\langle e \rangle$ and  
$\langle f(e)\rangle$  are  $\cF$-conjugate;
 \item[] $\cA(\cF)$ iff $f(U)$ is $\cF$-conjugate to $U$ for all $U \le E_1$.   
 \end{itemize}
Here by $\langle e\rangle$ we mean the subgroup generated by $e$. Note that  $f \in \cE(\cF)$ is equivalent to $f \in \Hom_\cF(E_1,E_2)$ and that $\varphi$ depends on the choice of $e$ for the categories $\cE'(\cF)$ and $\cE'_\RR(\cF)$.

Our results describe various spectra of subrings of $H^*(\cF)$ in
terms of the above categories. Assume that $k$ is algebraically
closed, and for a finitely generated commutative $\FF_p$-algebra $R$
write $V_R(k):=\Hom(R,k)$ for the variety of ring homomorphisms from
$R$ to $k$ with the Zariski topology generated by closed sets of
form $$\{\phi \in \Hom(R,k) \mid \ker(\phi) \supseteq I\},$$ for an
ideal $I \unlhd R$. Note that any ring homomorphism $f: R \rightarrow
R'$ determines a mapping of varieties $$f^*: V_{R'}(k) \rightarrow
V_R(k), \mbox{ given by } \phi \mapsto \phi \circ f.$$ Moreover, if $R'$
is a finitely generated $f(R)$-module, then $f^*$ has finite
fibres. Note also that there is a continuous map $$V_R(k) \rightarrow
\Spec(R), \mbox{ given by } \phi \mapsto \ker(\phi).$$ Depending on
the choice of $k$ this can be made surjective. Since $H^*(S)$ is a
graded finitely generated $\FF_p$-algebra, $h^*(S):=H^*(S)/\sqrt{0}$
is a finitely generated \textit{commutative} $\FF_p$-algebra (here
$\sqrt{0}$ denotes the ideal generated by elements which square to
$0$) and we write $$X_S(k):=V_{h^*(S)}(k)=\Hom(h^*(S),k)$$ for the
associated variety. Note that a group homomorphism $f: P \rightarrow
Q$ induces a continuous map $f_*: X_P(k) \rightarrow X_Q(k)$ between
the associated varieties.  With the above terminology, Linckelmann has
shown \cite{L17} that there exists a
homeomorphism $$\colim_{\substack{\cE(\cF)}} X_E(k) \rightarrow
V_R(k)$$ where $R=H^*(\cF) \subseteq H^*(S)$. This is an analogue of
Quillen's description of the spectrum of $H^*(G)$ mentioned above.  In
Section \ref{ss:coh2} we consider subrings of $H^*(\cF)$ generated by
Chern classes of $\cF$-stable ordinary representations of $S$: those
for which the associated character is constant on $\cF$-conjugacy
classes (see Section \ref{s:fstable}). Our first main result may be
viewed as an analogue of \cite[Proposition 7.1]{GL98} for fusion
systems:
\begin{thm}\label{t:main1}
    Let $\cF$ be a saturated fusion system on a finite $p$-group $S$
    and let $R$ be the subring of $H^*(\cF)$ generated by Chern
    classes of:
\begin{itemize}

    \item[(1)]  representations of $S$;
    \item[(2)]  real representations of $S$;
    \item[(3)]  permutation representations of $S$,

\end{itemize}
which are $\cF$-stable. Then in each case, there is a homeomorphism $$\colim_{\substack{\cC(R)}} X_E(k) \rightarrow V_R(k)$$  where the category $\cC(R)$ is:
$$(1) \ \cE'(\cF); \ (2) \ \cE_\RR'(\cF);  \ (3) \ \cE_P'(\cF).  $$
\end{thm}

To prove Theorem \ref{t:main1}, we first observe that in each case the
subring $R$ is both large and natural (see Definition
\ref{d:largenat}). We then apply a result of Green and the first named
author, Theorem \ref{t:glmain}, to deduce the existence of a category
$\cC(R)$ as in the conclusion of Theorem \ref{t:main1}. To describe
the morphisms in $\cC(R)$ we exploit the fact that they are uniquely
determined by how they interact with the characters of the
representations we consider (see Lemma \ref{l:glcid}). In case (3), we
rely on the existence of an explicit basis for the ring of
$\cF$-stable permutation characters of $S$ determined by Reeh
\cite{R15}.

In~\cite{GLS02} the authors study, for a finite group $G$, the variety
for the subring $\S(G)$ of $H^*(G)$ generated by the images of the
maps $\rho^*:H^*(\Sigma_n)\rightarrow H^*(G)$ (there the ring $\cS(G)$
was denoted by $S_h(G)$ which clashes with our use of $S$ as a Sylow
$p$-subgroup). Our second main result may be regarded as an analogue of
\cite[Theorem 2.6]{GLS02} for fusion systems:

\begin{thm}\label{t:main2}
    Let $\cF$ be a saturated fusion system on a finite $p$-group $S$
    and let $R$ be the subring of $H^*(S)$ generated by characteristic
    classes of $\cF$-stable permutations of $S$. Then there is a
    homeomorphism:
$$\colim_{\substack{\cA(\cF)}} X_E(k) \rightarrow V_R(k).$$  
\end{thm}

Our argument to prove Theorem \ref{t:main2} is an adaptation of that
found in \cite{GLS02}, for finite groups, and relies on particular
properties of Reeh's basis of $\cF$-stable permutation characters of
$S$. As for the subrings considered in Theorem
\ref{t:main1}, we show in Corollary \ref{cor:conj63} that the ring $R$ in
Theorem \ref{t:main2} can be described in terms of images in
cohomology of maps between classifying spaces.

We close the introduction with some remarks pertaining to a possible
extension of Theorems \ref{t:main1} and \ref{t:main2} to the case of
fusion systems on infinite groups. Indeed, the main result of
\cite{GL98} is concerned with varieties for the cohomology of
\textit{any} compact Lie group. The fusion system of a such a group is
a particular example of a \textit{$p$-local compact group} which is a
saturated fusion system on a discrete $p$-toral group (a $p$-group
with a finite index infinite torus $(\ZZ/p^\infty)^n$, for some $n \ge
1$). There is a version of Quillen stratification for such groups (see
\cite[Theorem 5.1]{BCHV19}), and it has been shown that certain
classes of $p$-local compact groups, for example those coming from
finite loop spaces and $p$-compact groups, admit unitary embeddings,
at least in one of two possible senses (see \cite{CC17}). Observe that
the existence of unitary embeddings for compact Lie groups is a key
ingredient in the proof of \cite[Proposition 2.2]{GL98}.

\textbf{Acknowledgements}: The authors thank Assaf Libman for drawing
their attention to~\cite{CL09}, which enabled them to resolve a
conjecture posed in an earlier version of this article.  
The second author gratefully acknowledges
funding from the UK Research Council EPSRC for the project
EP/W028794/1.  The first author visited Universit\"at Bielefeld while
working on this article, and gratefully acknowledges their
hospitality.  The visit was supported by a grant from the Deutsche
Forschungsgemeinschaft (DFG), Project-ID 491392403 -- TRR~358.

\section{$\cF$-stable representations and group actions}\label{s:fstable}

We adopt standard notation for fusion systems as found, for example, in \cite{AKO11}. Let $\cF$ be a saturated fusion system on a finite $p$-group $S$. 

\subsection{The ring of $\cF$-stable representations}
Recall that an object $P \le S$ is \textit{$\cF$-centric} if for all
morphisms $\varphi \in \Hom_\cF(P,S)$, we have $C_S(\varphi(P)) \le
\varphi(P)$. Let $\cF^c$ denote the set of $\cF$-centric subgroups.

\begin{defn}
The \textit{orbit category} $\cO=\cO(\cF)$ of $\cF$ is the category defined via:
\begin{itemize}
\item[(a)] $\ob(\cO)=\{P \mid P \le S \}$;
\item[(b)] for each $P,Q \le S$,
  $\Hom_\cO(P,Q)=\Rep_\cF(P,Q):=\Hom_\cF(P,Q)/\Inn(Q)$ is the set of
  $\Inn(Q)$-orbits of $\Hom_\cF(P,Q)$ (with action given by right
  composition of morphisms).
\end{itemize} 
The centric orbit category $\cO(\cF^c)$ is the full subcategory of
$\cO$ with object set $\cF^c$.
\end{defn}

\begin{defn}
An ordinary character $\chi$ of $S$ is \textit{$\cF$-stable} if for all $g
\in S$ and morphisms $\varphi \in \Hom_{\cF}(\langle g \rangle,S)$,
$\chi(\varphi(g))=\chi(g)$. That is, $\chi$ takes the same value on
all members of each $\cF$-conjugacy class of $S$. Denote by $\C(\cF)$
the subring of $\C(S)$ (the character ring of $S$) consisting of
$\cF$-stable characters. Also, for a natural number $n$, denote
by $\C_n(S)$ and $\C_n(\cF)$ the subsets of characters of degree $n$.
\end{defn}

Following \cite{CCM20} for any group $G$, let $\Rep_n(G)=\Rep(G,U(n))$
denote the set of isomorphism classes of $n$-dimensional ordinary
representations of $G$. Let $R(G)$ denote the representation ring of
$G$.

\begin{defn}
A complex representation $\rho$ of $S$ is
\textit{$\cF$-fusion preserving} if $\rho|_P=\rho|_{\varphi(P)} \circ
\varphi \in \Rep_n(P)$ for any $P \le S$ and $\varphi \in
\Hom_\cF(P,S)$; let $\Rep_n(\cF)$ denote the set of isomorphism
classes of $n$-dimensional complex $\cF$-fusion preserving
representations of $S$
\end{defn}

Note that $\rho \in \Rep_n(\cF)$ if and only if $\chi_\rho \in
\C_n(\cF)$ where $\chi_\rho$ is the character associated to
$\rho$. Using the Alperin-Goldschmidt fusion theorem for fusion
systems, one can show:

\begin{prop}\label{p:repfact} Let $\cF$ be a saturated fusion system over $S$.
  Then $$\lim_{\substack{\longleftarrow \\ \cO(\cF^c)}} \Rep_n(P) \cong \Rep_n(\cF).$$
\end{prop}

\begin{proof}
This is a straightforward modification of the argument used to prove \cite[Proposition 3.6]{CCM20} with  $\cF^{cr}$ replaced by $\cF^{c}$.
\end{proof}

Now let $R(\cF)$ be the subring of $\cF$-stable representations in
$R(S)$ and $\C(\cF)$ be the Grothendieck group
of $$\bigcup_{n=1}^\infty \C_n(\cF).$$ 

Write $S^\cF$ for a set of
$\cF$-conjugacy class representatives of $S$.

\begin{thm}\label{t:rankconj}
$\C(\cF) \otimes \CC$ is equal to the space of $\CC$-class functions  on $S^\cF$.
\end{thm}

\begin{proof}
See \cite[Lemma 2.1]{BC20}.
\end{proof}

From this result, we easily deduce that two elements are
$\cF$-conjugate if their character values coincide:

\begin{cor}\label{c:rankconj}
If $s,t \in S$ are such that $\chi(s)=\chi(t)$ for all $\chi \in
\C(\cF)$, then $s$ and $t$ are $\cF$-conjugate.
\end{cor}

\subsection{The ring of $\cF$-stable $S$-sets}
If $S$ acts on a finite set $X$ and $\phi:P\rightarrow S$ is a
homomorphism, denote by $^\phi X$ the $P$-set $X$ with action given  by $p\cdot x=\phi(p)x$, where the right side is the original action of $S$ on $X$.
\begin{defn}
Let $X$ be a finite $S$-set.
\begin{itemize}
    \item[(1)] $X$ is said to be \textit{$\cF$-stable} if for every
      $P\leq S$ and every morphism $\phi:P\rightarrow S$ in $\cF$ the
      $P$-sets $X$ and $^\phi X$ are isomorphic.
 \item[(2)] $X$ is said to be \textit{linearly $\cF$-stable} if the
   associated permutation character is $\cF$-stable.
\end{itemize}
   \end{defn}

Plainly any $\cF$-stable $S$-set is linearly $\cF$-stable, but the
converse is not true. For example \cite[Section 7]{GL98} discusses an
example of this phenomenon for $G=\GL(3,\FF_p)$.  Note that we may,
equivalently define a homomorphism $\rho:S\rightarrow \Sigma_n$
associated to an  $S$-set $X$ of cardinality $n$, to be $\cF$-stable if
for all $P\leq S$ and all morphisms $\phi:P\rightarrow S$ in $\cF$,
the morphisms $\rho|_P$ and $\rho\circ \phi$ differ by an inner
automorphism of $\Sigma_n$.

 If $X$ is an $S$-set and $Q \le S$, let
 $\Phi_Q(X)=|X^Q|$ denote the number of $Q$-fixed points of $X$. To
 prove Theorem \ref{t:main2} we shall need to know that there are
 sufficiently many $\cF$-stable permutation
 representations. Reeh~\cite{R15} shows the following:

\begin{prop}\label{p:reehmonoid}

For each $P \le S$ there exists an
$\cF$-stable $S$-set $\alpha_P$ with the properties that:
\begin{itemize}
    \item[(1)] $\Phi_Q(\alpha_P)=0$ unless $Q$ is $\cF$-subconjugate
      to $P$;
    \item[(2)] $\Phi_{P'}(\alpha_P)=|N_S(P')/P'|$ when $P'$ is a fully
      $\cF$-normalised $\cF$-conjugate of $P$; and 
      \item[(3)] $\alpha_P \cong \alpha_Q$ as $S$-sets if $P$ and $Q$ are $\cF$-conjugate.
\end{itemize}
\end{prop}

\begin{proof}
    See \cite[Proposition 4.8]{R15}.
\end{proof}

In fact, Reeh shows that the $S$-sets $\alpha_Q$ as $Q$ ranges over a set 
of $\cF$-class representatives of subgroups of $S$ can be chosen
to form an additive basis for the Burnside ring of $\cF$-stable
$S$-sets, but we will not need this.

\section{Cohomology of fusion systems and the Chern subring}
As in the previous section, we let $\cF$ be a saturated fusion system
on a finite $p$-group $S$.  As shown by Chermak \cite{C13}, to $\cF$ we may
associate a unique (up to isomorphism) centric linking system $\cL$
whose $p$-completed nerve plays the role of the classifying space of
$\cF$. In particular, when $\cF=\cF_S(G)$ is the fusion system of a
finite group with Sylow $p$-subgroup $S$, we have $|\cL|^\wedge_p
\simeq BG^\wedge_p$. The triple $(S,\cF, \cL)$ is sometimes referred
to as a \textit{$p$-local finite group.}
\subsection{The cohomology ring of a fusion system}\label{ss:coh1}

Following \cite[Section 5]{BLO03}, we define the cohomology of $\cF$ as follows:

\begin{defn}\label{d:homfstable}
The subring $H^*(\cF; \FF_p)$ of \textit{$\cF$-stable elements} of
$H^*(S,\FF_p)$ is the preimage in $H^*(S,\FF_p)$ of the natural
map $$H^*(S,\FF_p) \rightarrow \lim_{\substack{\longleftarrow
    \\ \cO(\cF^c)}} H^*(-; \FF_p).$$
\end{defn} From \cite[Theorem 4.2]{CCM20}, we obtain:

\begin{thm}\label{t:cohomology}
There is an isomorphism
$$H^*(|\cL|^\wedge_p; \FF_p) \cong \lim_{\substack{\longleftarrow
    \\ \cO(\cF^c)}} H^*(BP; \FF_p).$$ In particular, the rings
$H^*(\cF; \FF_p)$ and $H^*(|\cL|^\wedge_p; \FF_p)$ are isomorphic.
\end{thm}

In \cite[Theorem 5.3]{CCM20}, for $m > 0$, the authors
prove the following:

\begin{thm}\label{t:castetal}
There is a natural map $$\psi_m: [|\cL|^\wedge_p, BU(m)^\wedge_p]
\rightarrow \Rep_m(\cF)$$ satisfies: \begin{itemize}
    \item[(1)] for each $\rho \in \Rep_m(\cF)$ and sufficiently
      large $M>0$, $\rho \oplus M\reg \in \im(\psi_{m+M|S|})$;
    \item[(2)] if $f_1,f_2 \in [|\cL|^\wedge_p, BU(m)^\wedge_p] $ are
      such that $\psi_m(f_1)=\psi_m(f_2)$ then $f_1 \oplus h \simeq
      f_2 \oplus h$ for some $h \in [|\cL|^\wedge_p, BU(n)^\wedge_p]$
      with $\psi_n(h)=N\reg$ (some $N \ge 0$).
\end{itemize}
\end{thm}
Here $\reg$ denotes the regular representation of $S$ and $\oplus$ is
the Whitney sum.  Note that we have strengthened the \emph{statement}
of (1) above compared to that given in~\cite[Theorem 5.3]{CCM20};
there it is claimed only that there exists some $M>0$, but the
argument given proves the stronger claim which we will require.

\subsection{Chern classes of $\cF$-stable representations}\label{ss:coh2}
Write $$\FF_p[c_1,c_2,\ldots,c_n]=H^*(BU(n);\FF_p) \cong
\FF_p[x_1,\ldots,x_n]^{\Sigma_n},$$ where $c_i$ has degree $2i$ and
the isomorphism is given by sending $c_i$ to the $i$th symmetric
polynomial. For any finite group $P$, a unitary representation $\rho:
P \rightarrow U(n)$ induces a map $\hat{\rho}: BP \rightarrow BU(n)$
whose homotopy class depends on the conjugacy class of
$\rho$. We thus obtain a map $$\rho^*: H^*(BU(n); \FF_p)
\rightarrow H^*(BP;\FF_p) = H^*(P;\FF_p)$$ and so define the $i$th
\textit{Chern class} of $\rho$ to be $c_i(\rho):=\rho^*(c_i) \in
H^{2i}(P;\FF_p)$.

In particular, if $\rho \in \Rep_n(\cF) \subseteq \Rep_n(S)$ then for
each $1 \le i \le n$, we have $c_i(\rho) \in H^{2i}(S; \FF_p)$. In
fact, we have:
\begin{prop}\label{p:cohcomp}
For $\cF$ and $\rho$ as above, $c_i(\rho) \in H^{2i}(\cF; \FF_p)$ for
each $1 \le i \le n$.
\end{prop}
\begin{proof}
By Proposition \ref{p:repfact} we can regard
$\rho$  as a tuple  
 $$(\rho_P)_P \in \lim_{\substack{\longleftarrow \\ \cO(\cF^c)}}
\Rep_n(P)$$ given by an $\cO(\cF^c)$-compatible family of representations. For each $\cF$-centric subgroup $P$, restriction induces
commutative diagrams $$\begin{tikzcd} H^*(BU(n);\FF_p) \arrow[rd,
    "\rho^*"'] \arrow[r, "\rho|_P"] & H^*(P;\FF_p) \\ & H^*(S;\FF_p)
  \arrow[u]
\end{tikzcd} \mbox{ and } \begin{tikzcd}
H^*(BU(n);\FF_p) \arrow[rd, "\rho^*"'] \arrow[r, "(\rho|_P)"] &
\mbox{$\displaystyle\lim_{\substack{\longleftarrow \\ \cO(\cF^c)}}$}
H^*(BP;\FF_p) \\ & H^*(\cF;\FF_p) \arrow[u, "\cong"]
\end{tikzcd} $$
and thus using Definition \ref{d:homfstable}, we have
$c_i(\rho)=(c_i(\rho|_P))_{P} \in H^{2i}(\cF;\FF_p)$.
\end{proof} 
Thus for $\rho \in \Rep_n(\cF)$, it makes sense to define the
\textit{$i$th Chern class} of $\rho$ to be $c_i(\rho) \in H^*(\cF;
\FF_p)$. We may further define
$c_\bullet(\rho)=1+c_1(\rho)+\cdots+c_n(\rho)$ to be the \textit{total Chern
  class} of $\rho$. This definition is extended to virtual
representations by setting $c_i(-\rho)=\rho^*(c_i')$ where
$c_\bullet'=1+c_1'+c_2'+ \cdots$ is the unique power series in
$\FF_p[[c_1,\ldots, c_n]]$ satisfying $c_\bullet'c_\bullet=1$.  In particular, for
each $i$ it follows that $c_i(-\rho)$ is expressible as a polynomial
in the Chern classes $c_j(\rho)$ for $j\leq i$.

\begin{defn}\label{d:chernsub}
The \textit{Chern subring} $\Ch(\cF)$ of $H^*(\cF;\FF_p)$ is the
subring generated by the $c_i(\rho)$ for all $i$ and virtual
representations $\rho$, or equivalently for all representations
$\rho$.
\end{defn} 

There is an alternative definition of the Chern subring using the
classifying space for the linking system, which we will temporarily
denote by $\Ch'(\cF)$.  The mod-$p$ cohomology of $BU(n)^\wedge_p$ is
of course a polynomial ring $\FF_p[c_1,\ldots,c_n]$, and $\Ch'(\cF)$
is defined to be the subring of $H^*(|\cL|^\wedge_p;\FF_p)$ generated
by the images in cohomology of all maps $f:|\cL|^\wedge_p\rightarrow
BU(n)^\wedge_p$ for all $n\geq 1$.  Similarly, for a finite group $G$
define $\Ch'(G)$ to be the subring of $H^*(G;\FF_p)$ generated by
the images in cohomology of all maps $f:BG\rightarrow BU(n)$.

\begin{prop} \label{prop:chequal}
We have $\Ch'(G)=\Ch(G)$ and $\Ch'(\cF)=\Ch(\cF)$.  
\end{prop}

\begin{proof}
  Each $n$-dimensional representation $\rho$ gives rise to a map
  $B\rho:BG\rightarrow BU(n)$, and so $\Ch(G)\subseteq \Ch'(G)$.
  In general, not every map $f:BG\rightarrow BU(n)$ arises in this
  way (see~\cite[Example 1.18]{A78}), but this is true stably, in
  the sense that there is a virtual representation $\rho'$ whose
  Chern classes coincide with those of $f$, as shown
  in~\cite[Theorem~1.10]{A78}.  It follows that $\Ch(G)=\Ch'(G)$.  

  By the statement for groups we see that $\Ch(S)=\Ch'(S)$, from which
  it follows that $\Ch'(\cF)\subseteq \Ch(\cF)$.  It remains to establish the
  opposite inclusion.  The direct analogue of~\cite[Theorem~1.10]{A78} for
  fusion systems is not known.  
  Instead, we show that given any
  $\theta\in \Rep_n(\cF)$, there exists $N\geq n$ and
  $f\in[|\cL|^\wedge_p,BU(N)^\wedge_p]$ so that for each $i\leq n$,
  $c_i(\theta)=c_i(f^*)$.

  For any $n$, note that $c_\bullet(p\theta)=(c_\bullet(\theta))^p$, and
  so inductively one sees that $c_i(p^k\theta)$ can only be non-zero
  when $p^k$ divides $i$.  Now, let $\rho$ denote the regular
  representation of $S$, and pick $p^k>n$ sufficiently large so that
  $\theta\oplus p^k\rho$ is realized by a map $|\cL|^\wedge_p\rightarrow
  BU(N)^\wedge_p$, where $N=n+p^k|S|$ by Theorem \ref{t:castetal}.  Each
  Chern class of $\theta\oplus p^k\rho$ is contained in $\Ch'(\cF)$, and
  for $i\leq n$, $c_i(\theta\oplus p^k\rho)= c_i(\theta)$.
\end{proof}

Note that $\Ch(\cF)$ is finitely generated by \cite[Proposition 2.1]{GL98}.

\section{Varieties and Quillen stratification}

\subsection{The Green-Leary category of elementary abelian subgroups}
Following \cite[Section 6]{GL98}, to a subring $R$ of the cohomology
ring of a finite group can be associated a certain diagram $\cC(R)$ of
elementary abelian subgroups, and this is used to recover the maximal ideal
spectrum of $R$ under mild conditions.

\begin{defn}
Let $S$ be a finite group and $R$ be a subring of $H^*(S; \FF_p)$. Let
$\cC=\cC(R)$ be the category whose objects are the elementary abelian
subgroups of $S$, and where $f \in \Hom_\cC(E_1,E_2)$ if and only if the
corresponding diagram

\begin{equation}\label{e:comm}
\begin{tikzcd}
R  \arrow[d, "\res" ]
& R \arrow[d, "\res" ] \arrow[l, "\Id"] \\
h^*(E_1;\FF_p) 
&  h^*(E_2;\FF_p) \arrow[l, "f^*"]
\end{tikzcd}
\end{equation}
commutes.
\end{defn}

As in \cite[Section 6]{GL98}, we also define:

\begin{defn}\label{d:largenat}
Let $S$ be a finite group and $R$ be a subring of $H^*(S; \FF_p)$. 
\begin{itemize}
\item[(1)] $R$ is \textit{large} if it contains the Chern classes of
  a positive multiple of the regular representation of $S$;
\item[(2)] $R$ is \textit{natural} if it is generated by homogeneous
  elements and closed under the action of the Steenrod algebra.
\end{itemize}
\end{defn}

We remark that the definition of \emph{large} given above is a
simplification of the one used in~\cite{GL98}.  In the case when $S$
is not finite (e.g., $S$ a compact Lie group or a $p$-toral group), a
large subring is one that contains the Chern classes of a virtual
representation of non-zero degree whose restriction to every
elementary abelian $p$-subgroup of~$S$ is regular.

In \cite[Theorem 6.1]{GL98} the authors prove:
\begin{thm}\label{t:glmain}
Let $S$ be a finite group and $R$ be a subring of $H^*(S; \FF_p)$. If
$R$ is large and natural then the map
$$\colim_{\substack{\cC(R)}} X_E(k) \rightarrow V_R(k)$$  is a homeomorphism.
\end{thm}

Note that if $R=H^*(S; \FF_p)$ then Theorem \ref{t:glmain} is due to
Quillen \cite{Q71}.  One tool for describing $\cC(R)$ is \cite[Lemma 9.2] {GL98}, which we restate here for convenience:

\begin{lem}\label{l:glcid}
Let $S$ be a finite group, let $A$ be an additive subgroup of $R(S)$
containing the regular representation and let $R=R_A$ be the subring
of $H^*(S)$ generated by Chern classes of elements of $A$. Then $R$ is
large and natural. Furthermore $f: E_1 \rightarrow E_2$ is a morphism
in $\cC(R)$ if and only if $\chi(e) = \chi(f(e))$ for all $e \in E_1$ and all characters
$\chi$ of elements of $A$.
\end{lem}

For example, if $R$ is the subring generated by the Chern class of the
regular representation then $R$ is large and natural, and $\cC(R)$ is
the category of all injective maps between elementary abelian
subgroups by \cite[Lemma 6.2]{GL98}.

\subsection{Quillen stratification for fusion systems}
Now let $S$ is be a finite $p$-group and $\cF$ be a saturated fusion
system on $S$. We first apply Theorem \ref{t:glmain} to reinterpret
Linckelmann's description of the spectrum of the cohomology ring of a
fusion system.  Recall that a subgroup $P \le S$ is said to be
\textit{$\cF$-subconjugate} to $Q \le S$ if some $\cF$-conjugate of
$P$ is contained in $Q$.

\begin{prop}\label{p:linckelmannprop}
Let $\cF$ be a saturated fusion system on $S$ and $E$ be an elementary
abelian subgroup of $S$. Let $\sigma_E$ be a homogeneous element in
$h^*(E;\FF_p)$ satisfying $$\res^E_F(\sigma_E)=0 \mbox{ for each $F <
  E$.}$$ Then,
\begin{itemize}
    \item[(1)] for any $\eta \in h^*(E;\FF_p)^{\Aut_\cF(E)}$, there is
      $\eta ' \in H^*(\cF;\FF_p)$ such that $\res^S_E(\eta ')=(\sigma_E
      \cdot \eta)^{p^a}$; and
    \item[(2)] there exists $\rho_E \in h^*(\cF;\FF_p)$ such that
      $\res^S_E(\rho_E)=(\sigma_E)^{p^a}$ and $\res^S_F(\rho_E)=0$ for all
      subgroups $F$ to which $E$ is not $\cF$-subconjugate.
    \end{itemize}
\end{prop}

\begin{proof}
See \cite[Proposition 6]{L17}.
\end{proof}

Following \cite{L17}, let $X_\cF(k)=V_{h^*(\cF;\FF_p)}(k)$ denote the
maximal ideal spectrum of $H^*(\cF; \FF_p)$ and, for a subgroup $Q \le
S$, set $X_{\cF,Q}(k):=(\res^S_Q)^*(X_Q(k))$ where $\res_Q^S$ is the restriction map $H^*(\cF;\FF_p) \rightarrow H^*(Q;\FF_p)$ . Finally set $$X_Q^+(k):=X_Q(k) \backslash \bigcup_{R
  < Q} (\res_R^Q)^*(X_R(k)), \mbox{ and } X_{\cF,Q}(k)^+ =
r_Q^*(X_Q^+(k)).$$ The existence of an element $\sigma_E$ satisfying
the conditions in Proposition \ref{p:linckelmannprop} is shown in
\cite[Section 5.6]{B91} in the discussion which precedes \cite[Lemma
  5.6.2]{B91} and from this Linckelmann deduces in \cite[Theorem
  1(i)]{L17} that

$$X_\cF(k)=\bigcup_E X_{\cF,E}(k)=\coprod_E X^+_{\cF,E}(k),$$ is a
union of locally closed subvarieties, where $E$ runs through a set of
$\cF$-isomorphism class representatives of elementary abelian
subgroups of $S$. Equivalently, (c.f. \cite[Corollary 5.6.4]{B91}) we
have the following result:

\begin{thm}\label{t:linckisog}
The natural map $$\colim_{\substack{\cE(\cF)}}  X_E(k) \rightarrow X_\cF(k)$$
is an inseparable isogeny. 
\end{thm}

In particular, we have:

\begin{prop}\label{p:linckreint}
Suppose $R=H^*(\cF;\FF_p) \subseteq H^*(S;\FF_p)$. Then, 
\begin{itemize}
    \item[(1)] $R$ is large and natural; 
    \item[(2)] $\cC(R)$ is exactly $\cE(\cF)$.
\end{itemize}

\end{prop}

\begin{proof}
$R$ is large since the regular representation
  of $S$ is $\cF$-stable, and naturality follows
  because the action of the Steenrod algebra on the cohomology of each
  pair of subgroups $P,Q\leq S$ commutes with the maps induced by any
  homomorphism $\phi:P\rightarrow Q$ in $\cF$.  A consequence of the
  argument in \cite{L17} which proves Theorem \ref{t:linckisog} is
  that the morphisms $\varphi \in \Hom_\cF(E_1,E_2)$ are exactly those
  for which the diagram (\ref{e:comm}) commutes with
  $R=H^*(\cF;\FF_p)$, proving (2). 
\end{proof}

\section{Chern classes of $\cF$-stable representations} We now apply Theorem \ref{t:glmain} to describe the spectra of the subrings of $H^*(\cF)$ determined by various classes of $\cF$-stable representations of $S$. The reader is referred to Section \ref{s:intro} for definitions of the categories $\cE(\cF)$, $\cE'(\cF)$, $\cE_\RR'(\cF)$ and $\cE'_P(\cF)$. We start with the collection of all $\cF$-stable representations.

\begin{prop}\label{p:allreps}
If $R=\Ch(\cF) \subseteq H^*(\cF;\FF_p)$ then $\cC(R)=\cE'(\cF)$.
\end{prop}

\begin{proof}
Since $A=R(\cF) \subseteq R(S)$ is an additive subgroup of the
representation ring of $S$ generated by genuine representations, and
containing the regular representation, and since $\Ch(\cF)$ is exactly
the subring of $H^*(\cF;\FF_p)$ generated by Chern classes of elements
of $A$, we have by Lemma \ref{l:glcid} that: \begin{itemize}
\item[(1)] $R$ is large and natural;
\item[(2)] $f \in \Mor_{\cC(R)}(E_1,E_2)$ if and only if  $\chi(e)=\chi(f(e))$ for all $e
  \in E_1$ and all characters $\chi$ of elements of $A$.
   
\end{itemize}
Hence $\cC(R)=\cE'(\cF)$ by Corollary \ref{c:rankconj} and
so $$\colim_{\substack{\cE'(\cF)}} X_E(k) \rightarrow
V_{\Ch(\cF)}(k)$$ is a homeomorphism by Theorem \ref{t:glmain}.
\end{proof}

\subsection{Real representations}  The \textit{real Chern subring}
$\Ch_\RR(\cF)$ is defined to be the subring of $H^*(S)$ generated by
Chern classes of real $\cF$-stable representations of $S$.  There are
other possible definitions for $\Ch_\RR(\cF)$, such as the (possibly
larger) ring containing the Chern classes of all complex
representations whose characters are real, which we shall temporarily
denote by $\Ch''_\RR(\cF)$, and the (possibly smaller) ring generated
by the images of those maps $f:|\cL|^\wedge_p\rightarrow BU(n)^\wedge_p$ that factor
through the map $BO(n)^\wedge_p\rightarrow BU(n)^\wedge_p$ for some $n$, which we
temporarily denote by $\Ch'_\RR(\cF)$. We first show that all these
choices lead to homeomorphic varieties:

\begin{prop}
  The inclusions $\Ch'_\RR(\cF) \rightarrow \Ch_\RR(\cF)\rightarrow
  \Ch''_\RR(\cF)$ induce homeomorphisms
  $$V_{\Ch''_\RR(\cF)}(k) \rightarrow V_{\Ch_\RR(\cF)}(k) \rightarrow
  V_{\Ch'_\RR(\cF)}(k)$$   of the associated varieties.
\end{prop}

\begin{proof}
  There are inclusions $U(n)\rightarrow SO(2n)\rightarrow U(2n)$.
  If $M$ has trace $\lambda$, then the image of $M$ under this
  composite has trace $\lambda+\overline{\lambda}$.  Thus if
  $\rho$ is any complex representation with real character,
  then $2\rho$ is a real representation.  The
  total Chern characters are related by $c_\bullet(2\rho)=c_\bullet(\rho)^2$.
  In the case when $p\neq 2$, it follows that each $c_i(\rho)$ is
  in the subring generated by the Chern classes of $2\rho$, and
  so for $p\neq 2$, $\Ch''_\RR(\cF)=\Ch_\RR(\cF)$.  In the case
  when $p=2$, $c_{2i}(2\rho)=c_i(\rho)^2$, and since elements of
  the algebraically closed field $k$ of characteristic two have
  unique square roots, any ring homomorphism from $\Ch_\RR(\cF)$
  to $k$ extends uniquely to one from $\Ch''_\RR(\cF)$ to $k$.

  Let $\rho:S\rightarrow O(n)$ be a real representation of $S$
  that is $\cF$-stable.  The regular representation of $S$ is
  of course a real representation, 
  and so by the argument used in Proposition~\ref{prop:chequal}
  there exists  some  large $N$ and $f:|\cL|^\wedge_p\rightarrow BU(N)^\wedge_p$
  so that $c_i(f)=c_i(\rho)$ for $i\leq n$.  The composite
  of $f$ with the inclusion map $i:BU(N)^\wedge_p\rightarrow BSO(2N)^\wedge_p$
  has $c_\bullet(i\circ f)=(c_\bullet(f))^2$, and so by the argument used
  in the previous paragraph a ring homomorphism from $\Ch'_\RR(\cF)$
  to the algebraically closed field $k$ of characteristic $p>0$
  extends uniquely to a homomorphism from $\Ch_\RR(\cF)$ to $k$.
\end{proof}
We now obtain, just as in~\cite{GL98}, a description of the variety of
$\Ch_\RR(\cF)$.

\begin{prop}\label{p:realreps} If $R=\Ch_\RR(\cF) \subseteq H^*(\cF; \FF_p)$,
  then $\cC(R)=\cE'_\RR(\cF)$.
\end{prop}

\begin{proof}
  The regular real representation of $S$ is $\cF$-stable, and so
  $\Ch_\RR(\cF)$ is both large and natural by Lemma \ref{l:glcid}.  Moreover, it is
  easily seen to be natural and so Theorem~\ref{t:glmain} may be
  applied.  If $\chi$ is a real character, then $\chi(g)=\chi(g^{-1})$ for any $g\in S$. Hence $\cE'_\RR(\cF)$ is contained in
  $\cC(R)$.  By the argument given in \cite[Proposition 7.1 and
    Proposition 7.2]{GL98}, to establish the reverse inclusion it suffices to
  show that characters of real $\cF$-stable representations
  separate the conjugacy classes of pairs $\{g,g^{-1}\}$.  Let $g,h$
  be elements of $S$ and suppose that $h$ is not $\cF$-conjugate to
  either $g$ or to $g^{-1}$.  We need to construct a real $\cF$-stable
  character which takes different values on $g$ and $h$.  We first claim that
   there is an $\cF$-stable character $\chi$ with \begin{equation}\label{e:star}
  \chi(h)\neq \chi(g) \mbox{ and 
  } \chi(h) \neq \chi(g^{-1})=\overline{\chi(g)}.
 \end{equation} 
   To see this, note that Corollary \ref{c:rankconj} implies there are $\chi_1$ and $\chi_2$ with
$\chi_1(h)\neq \chi_1(g)$ and $\chi_2(h)\neq \chi_2(g^{-1})$.  Now at least one of the
three characters  $\chi_1$,  $\chi_2$, and  $\chi_1+\chi_2$ satisfies (\ref{e:star}). Indeed, the only way that $\chi_1$ can fail is if $\chi_1(h)=\chi_1(g^{-1})$ and the only
way that $\chi_2$ can fail is if $\chi_2(h)=\chi_2(g).$  In this case we have that
$\chi_1(h)+\chi_2(h)= \chi_1(h)+\chi_2(g) \neq \chi_1(g)+\chi_2(g)$ and
similarly
$\chi_1(h)+\chi_2(h)= \chi_1(g^{-1})+\chi_2(h) \neq \chi_1(g^{-1})+\chi_2(g^{-1})$.

   Now both the sum $\chi+\overline{\chi}$ and the
      product $\chi\overline{\chi}$ are $\cF$-stable real characters and we claim that at least one of these two characters will take
      different values on $g$ and $h$. If this is not the case then
      writing $\alpha=\chi(g)$ and $\beta=\chi(h)$, we
      have $$\beta\overline{\beta}=\alpha\overline{\alpha} \mbox{ and
      } \beta+\overline{\beta}=\alpha+\overline{\alpha}.$$ Hence
      $\{\chi(h),\overline{\chi}(h)\} =
      \{\chi(g),\overline{\chi}(g)\}$ is the solution set for the
      equation $$x^2-(\alpha+\overline{\alpha})x
      +\alpha\overline{\alpha} = x^2-(\beta+\overline{\beta})x
      +\beta\overline{\beta},$$ which contradicts $\chi(h) \notin \{\chi(g),\overline{\chi(g)}\}$, 

\end{proof}

\subsection{Permutation representations}
For a saturated fusion system $\cF$ on a $p$-group $S$, define the
\textit{linear permutation Chern subring} $\Ch_P(\cF)$ to be the
subring of $H^*(S)$ generated by the Chern classes of all $\cF$-stable
linear permutation representations. Define also the
\textit{permutation Chern subring} $\Ch'_P(\cF)$ to be the subring of
$H^*(S)$ generated by the linearizations of all $\cF$-stable
permutation representations. By Lemma \ref{l:glcid}, both these rings are large and natural and from the definitions,
$\Ch'_P(\cF)\subseteq \Ch_P(\cF)$.  One consequence of the next result
is that this inclusion induces a homeomorphism of varieties.

\begin{prop}\label{p:permreps} If $R=\Ch_P(\cF)$ and $R'=\Ch'_P(\cF)$,
  then $\cC(R)=\cC(R')=\cE'_P(\cF)$.
\end{prop}

\begin{proof}
Since $R'\subseteq R$, $\cC(R) \subseteq \cC(R')$.  Suppose $f: E_1
\rightarrow E_2$ is a morphism in $\cC(R')$ and let $e \in E_1$. By
Lemma \ref{l:glcid}, $\chi(e)=\chi(f(e))$ for all permutation
characters $\chi$ of $\cF$-stable $S$-sets.  

If $\langle e \rangle$ and $\langle f(e) \rangle$ are $\cF$-conjugate,
then $f(e)$ and $e^i$ are $\cF$-conjugate for some $i \ge 1$ and
$\chi(f(e))= \chi(e^i)=\chi(e)$ for all permutation characters $\chi$
of $\cF$-stable $S$-sets. Thus $\cE'_P (\cF) \subseteq \cC(R)$.

Now suppose that $\langle e \rangle$ and $\langle f(e) \rangle$ are
not $\cF$-conjugate, and let $\langle e' \rangle$ be a fully
$\cF$-normalised $\cF$-conjugate of $\langle e \rangle$.  We show that $f\notin \cC(R')$. Suppose, to the contrary, that $f\in \cC(R')$ and let $\chi$ be the permutation
character associated to the $\cF$-stable $S$-set $\alpha_{\langle e
  \rangle}$ from Proposition \ref{p:reehmonoid}.  Since $\langle f(e)
\rangle$ and $\langle e \rangle$ are not $\cF$-conjugate, we have that
$\Phi_{\langle f(e) \rangle}(\alpha_{\langle e \rangle})=0$ by
Proposition \ref{p:reehmonoid}(1). On the other hand, $$\Phi_{\langle
  e' \rangle}(\alpha_{\langle e \rangle})=|N_S(\langle e'
\rangle)/\langle e' \rangle| \neq 0$$ by Proposition
\ref{p:reehmonoid}(2). If $i \ge 1$ is such that $e'^i$ is
$\cF$-conjugate to $e$, then since $\chi$ is $\cF$-stable,
$$0 \neq |\Fix_{\alpha_{\langle e
    \rangle}}(e'^i)|=\chi(e'^i)=\chi(e)=\chi(f(e)) =
|\Fix_{\alpha_{\langle e \rangle}}(f(e))|.$$ Since $\langle f(e)
\rangle$ is cyclic of order $p$ this implies that $\Phi_{\langle f(e)
  \rangle}(\alpha_{\langle e \rangle}) \neq 0$, a contradiction.
  
We have thus shown that $\cE'_P (\cF)\subseteq \cC(R) \subseteq \cC(R') \subseteq \cE'_P (\cF)$, as required. 
\end{proof}

\begin{proof}[Proof of Theorem \ref{t:main1}]
  This follows on combining Propositions \ref{p:allreps}, \ref{p:realreps}
  and \ref{p:permreps}.
\end{proof}

\section{Characteristic classes of $\cF$-stable permutations}

In this section, $\cF$ is a saturated fusion system on a finite
$p$-group $S$. We adapt the discussion in \cite{GLS02} to the setting
of fusion systems.

\begin{lem}
    Let $X$ be an $\cF$-stable $S$-set and $\rho_X: S \rightarrow
    \Sigma_n$ be an associated homomorphism. Then the induced
    map $$\rho_X^*: H^*(\Sigma_n) \rightarrow H^*(S)$$ depends only on
    $X$ and has image contained in $H^*(\cF)$.
\end{lem}

\begin{proof}
    Since the images of two choices of $\rho_X$ differ only by an
    inner automorphism of $\Sigma_n$, $\rho_X^*$ depends only on
    $X$. Denote by $A_n(P)$ the set of isomorphism classes of $P$-sets
    of order $n$ for a finite group $P$, and by $A_n(\cF)$ the set of
    isomorphism classes of $\cF$-stable $S$-sets of order $n$. Then
    arguing exactly as in Proposition \ref{p:repfact}, we have a
    bijection $$\lim_{\substack{\longleftarrow \\ \cO(\cF^c)}} A_n(P)
    \cong A_n(\cF).$$ Now the argument in Proposition \ref{p:cohcomp}
    with $U(n)$ replaced by $\Sigma_n$ yields the result.
\end{proof}

\begin{defn}
Let $\cS(\cF)$ be the \textit{permutation subring} of $H^*(S)$
generated by $\im(\rho^*_X)$ for all $\cF$-stable $S$-sets.
\end{defn}  

Just as for the Chern subring, there is an analogous topological
definition: $\cS'(\cF)$ is defined to be the subring of
$H^*(|\cL|_p^\wedge;\FF_p)\leq H^*(S;\FF_p)$ generated by the images
in cohomology of all maps $f:|\cL|_p^\wedge\rightarrow
B(\Sigma_n)^\wedge_p$ for all $n\geq 1$.  A map $f$ as above defines
an $\cF$-stable $S$-set of cardinality $n$ and so there is an inclusion
map $\cS'(\cF)\rightarrow \cS(\cF)$.  We have been unable to
establish a direct analogue of Proposition~\ref{prop:chequal},
but we shall show below that this inclusion map induces an
isomorphism of varieties, a result that appeared as a conjecture in
an earlier version of this article.

The connection between $\cF$-stable $S$-sets and maps of $p$-completed
classifying spaces is rather more subtle than for $\cF$-stable unitary
representations; see the PhD thesis of Matthew Gelvin~\cite{G10} for a
discussion.  Nevertheless, in~\cite[Proposition~7.8]{CL09} it is shown
that for any $\cF$-stable $S$-set $X$ there is $k\geq 0$ and a map
$f:|\cL|_p^\wedge\rightarrow B(\Sigma_n)^\wedge_p$, where $n=p^k|X|$,
so that the $\cF$-stable $S$-set induced by $f$ is $p^k.X$.

   Recall the definition of the category $\cA=\cA(\cF)$ in Section
   \ref{s:intro} which is analogous to the category $\cA_h$ in
   \cite[Section 3]{GLS02} by \cite[Lemma 3.2]{GLS02}. We apply
   Proposition \ref{p:reehmonoid} to prove the following analogue of
   \cite[Lemma 2.7]{GLS02} for fusion systems:

\begin{lem}\label{l:sfident}
Let $E_1,E_2 \le S$ and $f: E_1 \rightarrow E_2$ be an injective group
homomorphism. Then $f \in \cA$ if and only if for every $x \in
\cS(\cF)$, the class $\res^S_{E_1}(x)-f^*\res^S_{E_2}(x)$ lies in the
nilradical of $H^*(E_1)$.  Moreover the same statement holds with
$\cS'(\cF)$ instead of $\cS(\cF)$.  
\end{lem}

\begin{proof}
    Suppose $f \in \cA$. Then $E_1$ and $f(E_1)$ are $\cF$-conjugate,
    so that by Proposition \ref{p:reehmonoid}(3), the $\cF$-stable
    $S$-sets $\alpha_{E_1}$ and $\alpha_{f(E_1)}$ are
    isomorphic. Denoting (respectively) by $\rho_1$ and $\rho_2$ the
    corresponding $S$-representations, we have $\rho_1|_{E_1} \cong
    \rho_2|_{f(E_1)} \circ f$ and so there exists some $\sigma \in
    \Sigma_{|S|}$ such that the
    diagram \begin{equation}\label{e:comm2}
\begin{tikzcd}
E_1 \arrow[d, "\rho_1|_{E_1}" ] \arrow[r, "f"] 
& f(E_1) \arrow[d, "\rho_2|_{f(E_1)}" ] \\
\Sigma_{|S|} \arrow[r, "c_{\sigma}"]
&  \Sigma_{|S|}
\end{tikzcd}
    \end{equation} commutes.
    Hence $\res^S_{E_1} - f^* \res^S_{E_2}$ kills $\im(\rho_1^*)$.
    Since $\cS'$ is a subring of $\cS$, this argument proves
    the `only if' part of the claim for $\cS'$ too.  

    Conversely, suppose that $f \notin \cA$. Then there exists $U \le
    E_1$ such that $U$ is not $\cF$-conjugate to $f(U)$. If $U$ is not
    fully $\cF$-normalised, then let $\varphi \in \Hom_\cF(U,S)$ be
    such that $\varphi(U)$ is fully $\cF$-normalised. Then, since
    $\cF$ is saturated, there exists $\widetilde{\varphi} \in
    \Hom_{\cF}(N_\varphi, S)$ which extends $\varphi$. Since $E_1 \le
    C_S(U) \le N_{\varphi}$ we obtain a map $$f':=f \circ
    (\widetilde{\varphi}|_{E_1})^{-1} \in \Hom(\varphi(E_1),E_2).$$
    Plainly $f' \notin \cA$ since $f \notin \cA$, so replacing $f$ by
    $f'$ and $U$ by $\varphi(U)$ if necessary, we may assume that $U$
    is fully $\cF$-normalised. Then by Proposition \ref{p:reehmonoid},
    $\Phi_{f(U)}(\alpha_{U})=0 \neq \Phi_U(\alpha_U)$.  Let $X$ be
    the direct sum of $\alpha_U$ and the regular representation.
    Then $\Phi_{f(U)}(X)=0\neq \Phi_U(X)$ and $S$ acts faithfully
    on $X$.  Let $\rho:S\rightarrow \Sigma_{|X|}$ be the corresponding
    permutation representation of $S$.  Then $\rho$ is injective and
    since $\Phi_{f(U)}(X)\neq \Phi_{U}(X)$, it follows that $\rho(U)$
    and $\rho(f(U))$ are not conjugate in the symmetric group
    $\Sigma_{|X|}$.  As in~\cite[Lemma~2.7]{GLS02}, the construction
    used in the proof of~\cite[Theorem~8.1]{GL98} gives a class
    $\zeta\in H^*(\Sigma_{|X|})$ whose image in $\rho(U)$ is
    non-nilpotent, while its image in $\rho(f(U))$ is zero.
    If we let $x=\rho^*(\zeta)$, then $\res^S_{E_1}(x)-f^*\res^S_{E_2}(x)$
    is not in the nilradical of $H^*(E_1)$, since its image under
    $\res^{E_1}_U$ is not nilpotent.  
   
    It cannot be assumed that there is a map $|\cL|_p^\wedge\rightarrow
    (B\Sigma_{|X|})_p^\wedge$ that gives rise to the $\cF$-stable $S$-set
    $X$.  However, by~\cite[Proposition~7.8(a)]{CL09} there exists
    $k\geq 0$ and a map
    $\rho':|\cL|_p^\wedge\rightarrow (B\Sigma_{p^k|X|})_p^\wedge$ that
    gives rise to the $\cF$-stable $S$-set $Y=p^k.X$.  As before,
    $\rho'$ defines an embedding of $S$ into $\Sigma_{|Y|}$ and
    since $\Phi_{U}(Y)\neq \Phi_{f(U)}(Y)$ it follows that $\rho'(U)$
    and $\rho'(f(U))$ are non-conjugate elementary abelian subgroups
    of $\Sigma_{|Y|}$ of the same rank.  Hence there is $\zeta'\in
    H^*(\Sigma_{|Y|})$ whose image in $\rho'(U)$ is non-nilpotent,
    while its image in $\rho'(f(U))$ is zero.  Now $x'=\rho'^*(\zeta')$
    has the property that $\res^S_{E_1}(x')-f^*\res^S_{E_2}(x')$ is not
    in the nilradical of $H^*(E_1)$ as before.
\end{proof}

We now obtain.

\begin{thm}\label{t:perms}
  The restriction maps in cohomology induce natural homeomorphisms
  $$\colim_{\substack{\cA(\cF)}}  X_E(k) \rightarrow V_{\cS(\cF)}(k)
  \rightarrow V_{\cS'(\cF)}(k).$$  
\end{thm}

\begin{proof}
    The regular representation $\rho: S \rightarrow \Sigma_{|S|}$ is
    obviously $\cF$-stable so $\cS(\cF)$ is large. It is also natural
    because $\cS(\cF)$ is clearly homogeneously generated and closed
    under the action of the Steenrod algebra. Hence by Theorem
    \ref{t:glmain}, there exists some category $\cC$ of elementary
    abelian $p$-subgroups for which $$\colim_{\substack{\cC}} X_E(k)
    \rightarrow V_{\cS(\cF)}(k)$$ is a homeomorphism.  Finally, Lemma \ref{l:sfident}
    identifies $\cC$ with the category $\cA(\cF)$ defined above.

    The subring $\cS'(\cF)$ is also clearly natural.  The Chern classes
    of the regular representation $\rho$ may not lie in $\cS'(\cF)$,
    but by~\cite[Proposition~7.8(a)]{CL09}, there exists $k$ so that
    the Chern classes of $p^k.\rho$ lie in $\cS'(\cF)$.  The remainder
    of the argument using Theorem~\ref{t:glmain} and Lemma~\ref{l:sfident}
    proceeds exactly as for $\cS(\cF)$.  
\end{proof}

\begin{cor}\label{cor:conj63}
  The inclusion $\cS'(\cF)\rightarrow \cS(\cF)$ is an inseparable isogeny.  
\end{cor}

\begin{proof}
  Immediate from Theorem~\ref{t:perms}, since this inclusion induces a
  homeomorphism of varieties.
\end{proof}


\begin{thebibliography}{99}

\bibitem{A78}
  {\sc J. F. Adams}, Maps between classifying spaces II. \emph{Inventiones math. \bf49} (1978), 1--65.



\bibitem{AKO11}
{\sc M. Aschbacher, R. Kessar, and B. Oliver}, \emph{Fusion Systems in Algebra
  and Topology}. London Mathematical Society Lecture Note Series, 391, Cambridge
  University Press, Cambridge, 2011.


\bibitem{BC20}
  {\sc N. B\'{a}rcenas and J. Cantarero}, A completion theorem for fusion
  systems. \emph{Israel Journal of Mathematics \bf236} (2020), 501--531.

  \bibitem{BCHV19}
    {\sc T. Barthel, N. Castellana, D. Heard  and G. Valenzuela}, Stratification
    and duality for homotopical groups.
  \emph{Advances in Mathematics \bf354} (2019), 1--61.
  

  
\bibitem{B91}
  {\sc D. J. Benson}, Representations and Cohomology, Vol. II: Cohomology
  of Groups and Modules.
  \emph{Advanced Mathematics. \bf31} (1991), Cambridge University Press.



\bibitem{BLO03}
{\sc C. Broto, R. Levi, and B. Oliver}, The homotopy theory of fusion systems.
  \emph{J. Amer. Math. Soc. \bf16} (2003), 779--856.

\bibitem{CC17}
  {\sc J. Cantarero and N. Castellana}, Unitary embeddings of finite loop
  spaces. \emph{Forum Mathematicum \bf29} (2017), 287--311.


\bibitem{CCM20}
  {\sc J. Cantarero, N. Castellana and L. Morales}, Vector bundles over
  classifying spaces of $p$-local finite groups and Benson--Carlson
  duality. \emph{J. London Math. Soc. \bf101} (2020), 1--22.

\bibitem{CL09}
  {\sc N. Castellana and A. Libman}, Wreath products and representations
  of $p$-local finite groups.  \emph{Adv. Math. \bf 221} (2009) 1302--1344.  
  

\bibitem{C13}
  {\sc A. Chermak}, Fusion systems and localities.  \emph{Acta Math. \bf 211} (2013) 47--139.  
 
\bibitem{G10}
{\sc M. J. K. Gelvin}, Fusion Action Systems. 
\emph{Ph. D. Thesis Mass. Inst. Tech.} (2010), arXiv:1008.1454.  



\bibitem{GL98} 
  {\sc D. J. Green and I. J. Leary}, The spectrum of the Chern
  subring. \emph{Commentarii Mathematici Helvetici \bf73} (1998), 406--426.

\bibitem{GLS02}
  {\sc D. J. Green, I. J. Leary and B. Schuster}, The subring of group
  cohomology constructed by permutation representations. \emph{Proc. Edinb.
  Math. Soc. (2) \bf45} (2002), 241--253.  
  

\bibitem{L17}
  {\sc M. Linckelmann}, Quillen's stratification for fusion
  systems. \emph{Comm. Algebra \bf45} (2017), 5227--5229.

\bibitem{Q71}
{\sc D. Quillen}, The spectrum of an equivariant cohomology 
ring. I, II. \emph{Ann. of Math. \bf 94} (1971), 549--572 and 573--602.

\bibitem{R15}
  {\sc S. P. Reeh}, The abelian monoid of fusion-stable finite sets is free.
  \emph{Algebra Number Theory \bf9} (2015), 2303--2324.





\end{thebibliography}
\end{document}